\documentclass[11pt]{amsart}
\usepackage{amsthm}
\usepackage{amssymb}
\usepackage{amsmath}
\usepackage{comment}
\usepackage{subcaption}

\usepackage{enumerate}

\usepackage{color}
\newcommand{\blue}{\textcolor{black}}

\usepackage{hyperref}
\usepackage[capitalise,noabbrev]{cleveref} 

\newtheorem{lemma}{Lemma}[section]
\newtheorem{theorem}[lemma]{Theorem}

\newtheorem{proposition}[lemma]{Proposition}
\newtheorem{conjecture}[lemma]{Conjecture}
\theoremstyle{definition}

\crefformat{sublemma}{#2(#1)#3}
\Crefformat{sublemma}{#2(#1)#3}
\crefmultiformat{sublemma}{#2(#1)#3}{ and~#2(#1)#3}{, #2(#1)#3}{ and~#2(#1)#3}

\DeclareMathOperator{\cl}{cl}

\begin{document}

\title[Cyclic Arrangements of Circuits and Cocircuits]{Matroids with a Cyclic Arrangement of Circuits and Cocircuits}

\author{Nick Brettell}
\address{School of Mathematics and Statistics, Victoria University of Wellington, New Zealand}
\email{nbrettell@gmail.com}

\author{Deborah Chun}
\address{Mathematics Department, West Virginia University Institute of Technology, 410 Neville Street, Beckley, West Virginia 25801, USA}
\email{deborah.chun@mail.wvu.edu}

\author{Tara Fife}
\address{Department of Mathematics, Louisiana State University, Baton Rouge, Louisiana, USA}
\email{tfife2@lsu.edu}

\author{Charles Semple}
\address{School of Mathematics and Statistics, University of Canterbury, Christchurch, New Zealand}
\email{charles.semple@canterbury.ac.nz}



\thanks{The authors thank the Mathematical Research Institute (MATRIX), Creswick, Victoria, Australia, for its support and hospitality during the Tutte Centenary Retreat 26 November--2 December 2017, where work on this paper was initiated. The first and fourth authors were supported by the New Zealand Marsden Fund.}

\date{\today}

\begin{abstract}
  For all positive integers $t$ exceeding one, a matroid has the \emph{cyclic $(t-1, t)$-property} if its ground set has a cyclic ordering $\sigma$ such that every set of $t-1$ consecutive elements in $\sigma$ is contained in a $t$-element circuit and $t$-element cocircuit. We show that if $M$ has the cyclic $(t-1,t)$-property and $|E(M)|$ is sufficiently large, then 
these $t$-element circuits and $t$-element cocircuits are arranged in a prescribed way in $\sigma$,
which, for odd $t$, is analogous to how $3$-element circuits and cocircuits appear in wheels and whirls, and, for even $t$, is analogous to how $4$-element circuits and cocircuits appear in swirls.
Furthermore, we show that any appropriate concatenation~$\Phi$ of $\sigma$ is a flower. If $t$ is odd, then $\Phi$ is a daisy, but if $t$ is even, then, depending on $M$, it is possible for $\Phi$ to be either an anemone or a daisy.
\end{abstract}

\maketitle

\section{Introduction}

Wheels and whirls are matroids with the property that every element is in a $3$-element circuit and a $3$-element cocircuit. As a consequence of this property, no single-element deletion or single-element contraction of a wheel or whirl with rank at least three is $3$-connected, and Tutte's Wheels-and-Whirls Theorem establishes that these are the only $3$-connected matroids for which this holds~\cite{wtt}.

In fact, wheels and whirls have a stronger property concerning $3$-element circuits and $3$-element cocircuits. Let $M$ be a rank-$r$ wheel or rank-$r$ whirl, where $r\ge 2$. Then there is a cyclic ordering $(e_1, e_2, \ldots, e_{2r})$ on the elements of $M$ such that, for all odd $i\in \{1, 2, \ldots, 2r\}$, we have that $\{e_i, e_{i+1}, e_{i+2}\}$ is a $3$-element circuit and $\{e_{i+1}, e_{i+2}, e_{i+3}\}$ is a $3$-element cocircuit, where subscripts are interpreted modulo~$2r$. In particular, $M$ has the property that there is a cyclic ordering of $E(M)$ such that every consecutive pair of elements in this ordering is contained in a $3$-element circuit and a $3$-element cocircuit. In this paper, we investigate generalisations of this property.

Let $t$ be a positive integer exceeding one. A matroid $M$ has the {\em cyclic $(t-1, t)$-property} if there is a cyclic ordering $\sigma$ of $E(M)$ such that every $t-1$ consecutive elements of $\sigma$ is contained in a $t$-element circuit and a $t$-element cocircuit, in which case, $\sigma$ is a {\em cyclic $(t-1, t)$-ordering} of $M$.

Wheels and whirls have the cyclic $(2, 3)$-property. Two classes of matroids that have the cyclic $(3, 4)$-property are the familiar classes of spikes and swirls. For all $r\ge 3$, a {\em rank-$r$ spike} is a matroid $M$ on $2r$ elements whose ground set can be partitioned $(L_1, L_2, \ldots, L_r)$ into pairs such that, for all distinct $i, j\in \{1, 2, \ldots, r\}$, the union $L_i\cup L_j$ is a $4$-element circuit and a $4$-element cocircuit. Therefore, if $\sigma$ is a cyclic ordering of $E(M)$ such that, for all $i$, the two elements in $L_i$ are consecutive in $\sigma$, then $\sigma$ is a cyclic $(3, 4)$-ordering of $M$.
For all $r\ge 3$, a {\em rank-$r$ swirl} is a matroid $M$ on $2r$ elements obtained by taking a simple matroid whose ground set is the disjoint union of a basis $B=\{b_1, b_2, \ldots, b_r\}$ and $2$-element sets $L_1, L_2, \ldots, L_r$ such that $L_i\subseteq \cl(\{b_i, b_{i+1}\})$ for all $i\in [r]$, where subscripts are interpreted modulo $r$, and then deleting $B$. Now let $\sigma=(e_1, f_1, e_2, f_2, \ldots, e_r, f_r)$, where $L_i=\{e_i, f_i\}$ for all $i$. Then $L_i\cup L_{i+1}$ is a $4$-element circuit and a $4$-element cocircuit for all $i$, so $\sigma$ is a cyclic $(3, 4)$-ordering of $M$.

%

If a matroid $M$ has the cyclic $(1, 2)$-property, then it is easily checked that $M$ is obtained by taking direct sums of copies of $U_{1, 2}$. However, if $t\ge 3$, then matroids with the cyclic $(t-1, t)$-property are highly structured. For example, suppose $t=3$, and let $(e_1, e_2, \ldots, e_{2r})$ be a cyclic $(2, 3)$-ordering of the rank-$r$ wheel, where $r\ge 4$. Then, for all $i\in \{1, 2, \ldots, 2r\}$, there is a unique $3$-element circuit and a unique $3$-element cocircuit containing $\{e_i, e_{i+1}\}$. Up to parity, the circuit is $\{e_i, e_{i+1}, e_{i+2}\}$ and the cocircuit is $\{e_{i-1}, e_i, e_{i+1}\}$. The first main result of the paper, Theorem~\ref{main1}, extends this to all positive integers $t$.


\begin{theorem}
\label{main1}
Let $M$ be a matroid and suppose that $\sigma=(e_1, e_2, \ldots, e_n)$ is a cyclic $(t-1, t)$-ordering of $E(M)$, where $n\ge 6t-10$ and $t\ge 3$. Then $n$ is even and, for all $i\in [n]$, there is a unique $t$-element circuit and a unique $t$-element cocircuit containing $\{e_i, e_{i+1}, \ldots, e_{i+t-2}\}$. Moreover,
\begin{enumerate}[{\rm (I)}]
\item If $t$ is odd, then the following hold:
\begin{enumerate}[{\rm(i)}]
\item For all $i\in [n]$, the subset $\{e_i, e_{i+1}, \ldots, e_{i+t-1}\}$ is either a $t$-element circuit or a $t$-element cocircuit, but not both.

\item For all $i\in [n]$, the subset $\{e_i, e_{i+1}, \ldots, e_{i+t-1}\}$ is a $t$-element circuit if and only if $\{e_{i+1}, e_{i+2}, \ldots, e_{i+t}\}$ is a $t$-element cocircuit.

\item For all $j\equiv i \bmod{2}$, if $\{e_i, e_{i+1}, \ldots, e_{i+t-1}\}$ is a $t$-element circuit, then $\{e_j, e_{j+1}, \ldots, e_{j+t-1}\}$ is a $t$-element circuit.
\end{enumerate}

\item If $t$ is even, then the following hold:
\begin{enumerate}[{\rm (i)}]
\item For all $i\in [n]$, exactly one of $\{e_i, e_{i+1}, \ldots, e_{i+t-1}\}$ and $\{e_{i+1}, e_{i+2}, \ldots, e_{i+t}\}$ is a $t$-element circuit.

\item For all $i\in [n]$, the subset $\{e_i, e_{i+1}, \ldots, e_{i+t-1}\}$ is a $t$-element circuit if and only if it is a $t$-element cocircuit.

\item For all $j\equiv i \bmod{2}$, if $\{e_i, e_{i+1}, \ldots, e_{i+t-1}\}$ is a $t$-element circuit, then $\{e_j, e_{j+1}, \ldots, e_{j+t-1}\}$ is a $t$-element circuit.
\end{enumerate}
\end{enumerate}
\end{theorem}

Motivated by Theorem~\ref{main1} and, in particular, the way consecutive elements in a cyclic $(t-1, t)$-ordering of a matroid are arranged as $t$-element circuits and $t$-element cocircuits, we next consider the following class of matroids. Let $M$ be a matroid with $n=|E(M)|$ and let $t$ be a positive integer such that $n\ge t+1$. We call $M$ {\em $t$-cyclic} if there exists a cyclic ordering $\sigma=(e_1, e_2, \ldots, e_n)$ of $E(M)$ such that, for all odd $i\in \{1, 2, \ldots, n\}$, either
\begin{enumerate}[(i)]
\item $\{e_i, e_{i+1}, \ldots, e_{i+t-1}\}$ is a $t$-element circuit and $\{e_{i+1}, e_{i+2}, \ldots, e_{i+t}\}$ is a $t$-element cocircuit, or

\item $\{e_i, e_{i+1}, \ldots, e_{i+t-1}\}$ is a $t$-element circuit and $t$-element cocircuit.
\end{enumerate}
If $\sigma$ is such an ordering of $E(M)$, then $\sigma$ is a {\em $t$-cyclic ordering} of $M$, in which case $\sigma$ is {\em odd} if it satisfies (i) and {\em even} if it satisfies (ii).

It is easily seen that wheels and whirls are $3$-cyclic.
The $(3, 4)$-cyclic orderings of spikes and swirls stated earlier are also $4$-cyclic orderings, so spikes and swirls are $4$-cyclic. Moreover, it follows from Theorem~\ref{main1} that if a matroid $M$ has the cyclic $(t-1, t)$-property for some positive integer $t$ exceeding one, then $M$ is $t$-cyclic provided $|E(M)|$ is sufficiently large.

In the second half of the paper, we establish properties of $t$-cyclic matroids. As well as showing basic properties such as the rank and corank of a $t$-cyclic matroid are equal for all $t$, we prove the next two theorems which show that $t$-cyclic matroids naturally give rise to flowers. For the reader unfamiliar with flowers, the notation and terminology relevant to these theorems are given in Section~\ref{prelim}.

The parity of $t$ impacts the structure of a $t$-cyclic matroid.  We first consider the case where $t$ is odd.


\begin{theorem}
\label{oddflower}
Let $t$ be a positive odd integer exceeding one, and let $M$ be a matroid. Suppose that $\sigma$ is an odd t-cyclic ordering of $M$. If $\Phi=(P_1, P_2, \ldots, P_m)$ is a concatenation of $\sigma$ with $|P_i|\ge t-1$ for all $i\in [m]$, then $\Phi$ is a $t$-daisy. Moreover, for all $i\in [m]$, we have $\sqcap(P_i, P_{i+1})=\frac{1}{2}(t-1)$ and, for all non-consecutive petals $P_i$ and $P_j$, we have $\sqcap(P_i, P_j)\le \frac{1}{2}(t-3)$.
\end{theorem}

\noindent Note that if $M$ is a $1$-cyclic matroid with $n$ elements, then it is easily seen that $M$ is the (disjoint) union of $\frac{n}{2}$ loops and $\frac{n}{2}$ coloops. Therefore, any cyclic ordering of $E(M)$, where every two consecutive elements consists of a loop and a coloop, is a $1$-cyclic ordering of $M$. Furthermore, if $\sigma$ is such an ordering and $\Phi$ is a concatenation of $\sigma$ into non-empty sets, then $\Phi$ is a $1$-anemone.

We obtain the following when $t$ is even.

\begin{theorem}
\label{evenflower}
Let $t$ be a positive even integer and let $M$ be a matroid. Let $\sigma=(e_1, e_2, \ldots, e_n)$ be an even $t$-cyclic ordering of $E(M)$, and suppose that $\Phi=(P_1, P_2, \ldots, P_m)$ is a concatenation of $\sigma$ such that, for all $i\in [m]$, if
$$P_i=\{e_{j+1}, e_{j+2}, \ldots, e_{j+k}\},$$
then $|P_i|\ge t-2$, $|P_i|$ is even, and $j+1$ is odd. Then $\Phi$ is a $(t-1)$-flower. Moreover, for all $i\in [n]$, we have $\sqcap(P_i, P_{i+1})=\frac{1}{2}(t-2)$.
\end{theorem}

\noindent In reference to Theorem~\ref{evenflower}, observe that we have not specified whether $\Phi$ is a $(t-1)$-anemone or a $(t-1)$-daisy. If $t=2$, then $\Phi$ is a $1$-anemone.  However, for all even $t\ge 4$, there exist $t$-cyclic matroids giving rise to $(t-1)$-anemones and $t$-cyclic matroids giving rise to $(t-1)$-daisies.  
This follows from a construction that obtains, for all $t\ge 2$, a $(t+2)$-cyclic matroid from a $t$-cyclic matroid.
Indeed, we conjecture that for all even $t\ge 4$, every $t$-cyclic matroid can be constructed from a $4$-cyclic matroid that is either a spike or a swirl by a generalisation of this construction. A more precise statement of this conjecture is given at the end of the paper.

Matroids with the property that every $t$-element subset of the ground set is contained in both an $\ell$-element circuit and an $\ell$-element cocircuit have recently been studied~\cite{tspikes}, continuing similar investigations in \cite{miller2014,pfeil}.
In particular, sufficiently large matroids with the property that every $t$-element set is contained in a $2t$-element circuit and $2t$-element cocircuit have a partition into pairs such that the union of any $t$ pairs is a circuit and a cocircuit.
For such matroids, 
there is an obvious cyclic ordering of the ground set that demonstrates these are $2t$-cyclic matroids.

The paper is organised as follows. The next section consists of some preliminaries, while Section~\ref{proof1} consists of the proof of Theorem~\ref{main1}. Basic properties of $t$-cyclic matroids are established in Section~\ref{tcyclic}, and the proofs of Theorems~\ref{oddflower} and~\ref{evenflower} are given in Section~\ref{flowers}. Lastly, in Section~\ref{build}, we detail, for all $t\ge 2$, a construction that produces a $(t+2)$-cyclic matroid from a $t$-cyclic matroid. We will use this construction to show that, for all even $t\ge 4$, 
there are $t$-cyclic matroids that give rise to $(t-1)$-anemones, and $t$-cyclic matroids that give rise to $(t-1)$-daisies.

\section{Preliminaries}
\label{prelim}

Notation and terminology follows Oxley~\cite{oxbook}, and the phrase ``by orthogonality'' refers to the fact that a circuit and cocircuit of a matroid cannot intersect in exactly one element. We use $[n]$ to denote the set $\{1, 2, \ldots, n\}$. 
When $i \le j$, we use $[i,j]$ to denote the set $\{i,i+1,i+2,\dotsc,j\}$; whereas when $i > j$, we use $[i,j]$ to denote $[i,n] \cup [1,j]$.
If $\sigma=(e_1, e_2, \ldots, e_n)$ is a cyclic ordering of $\{e_i: i\in [n]\}$, then all subscripts are interpreted modulo $n$. Furthermore, we say that $(P_1, P_2, \ldots, P_m)$ is a {\em concatenation of $\sigma$} if there are indices
$$1\le k_1 < k_2 < \cdots < k_m\le n$$
such that $P_i=\big\{e_j: j\in [k_{i-1}, k_i-1]\big\}$ for all $i\in [m]$.
The following well-known lemma is used throughout the paper.


\begin{lemma}
\label{handy}
Let $e$ be an element of a matroid $M$, and let $X$ and $Y$ be disjoint sets that partition $E(M)-e$. Then $e\in \cl(X)$ if and only if $e\not\in \cl^*(Y)$.
\end{lemma}

\noindent {\bf Connectivity.} Let $M$ be a matroid with ground set $E$. The {\em connectivity function $\lambda$} of $M$ is defined, for all subsets $X$ of $E$, by
$$\lambda(X)=r(X)+r(E-X)-r(M).$$
Equivalently, for all subsets $X$ of $E$, we have $\lambda(X)=r(X)+r^*(X)-|X|$. A set $X$ or a partition $(X, E-X)$ is {\em $k$-separating} if $\lambda(X) < k$. Additionally, if $\lambda(X)=k-1$, then the $k$-separating set $X$ or $k$-separating partition $(X, E-X)$ is {\em exact}.

For all subsets $X$ and $Y$ of $E$, the {\em local connectivity} between $X$ and $Y$, denoted $\sqcap(X, Y)$, is defined by
$$\sqcap(X, Y)=r(X)+r(Y)-r(X\cup Y).$$
Note that $\sqcap(X, Y)=\sqcap(Y, X)$. Also, if $(X, Y)$ is a partition of $E$, then $\sqcap(X, Y)=\lambda(X)$.

\noindent {\bf Flowers.} Flowers naturally describe crossing separations in a matroid. Originally defined for $3$-separations in $3$-connected matroids~\cite{Oxley2004}, flowers were later generalised in order to describe crossing $k$-separations in a matroid, without any connectedness condition~\cite{Aikin2008}.

For a matroid $M$ and an integer $m>1$, a partition $\Phi=(P_1, P_2, \ldots, P_m)$ of $E(M)$ into non-empty sets is a {\em $k$-flower} with {\em petals} $P_1, P_2, \ldots, P_m$ if each $P_i$ is exactly $k$-separating and, when $m\ge 3$, each $P_i\cup P_{i+1}$ is exactly $k$-separating, where all subscripts are interpreted modulo $m$. It is also convenient to view $(E(M))$ as  $k$-flower with a single petal. Suppose $\Phi=(P_1, P_2, \ldots, P_m)$ is a $k$-flower of $M$. Then $\Phi$ is a {\em $k$-anemone} if $\bigcup_{i\in I} P_i$ is exactly $k$-separating for all proper subsets $I$ of $[m]$. Furthermore, $\Phi$ is a {\em $k$-daisy} if $\bigcup_{i\in I} P_i$ is exactly $k$-separating for precisely the proper subsets $I$ of $[m]$ whose members form a consecutive set in the cyclic order $(1, 2, \ldots, m)$. Aikin and Oxley~\cite[Theorem~1.1]{Aikin2008} showed that every $k$-flower of $M$ is either a $k$-daisy or a $k$-anemone.

Suppose that $\Phi=(P_1, P_2, \ldots, P_m)$ is a $k$-flower of a matroid $M$, where $m \ge 4$ and $\sqcap(P_i,P_{i+1}) = c$ for all $i \in [m]$.
To show that $M$ is a $k$-daisy, it suffices, by \cite[Lemma~4.3]{Aikin2008}, to show that $\sqcap(P_i, P_j) \neq c$ for some distinct $i,j \in [m]$.

\section{Proof of Theorem~\ref{main1}}
\label{proof1}

Throughout this section, let $M$ be a matroid and let $\sigma=(e_1, e_2, \ldots, e_n)$ be a cyclic $(t-1, t)$-ordering of $E(M)$, where $t\ge 3$. Furthermore, for all distinct $i, j\in [n]$, let $\sigma_{[i,\, j]}$ denote the set of elements $\{e_i, e_{i+1}, \ldots, e_j\}$ and let \blue{$X_i=\sigma_{[i,\, i+t-2]}$}. If $C_i$ (resp.\ $C^*_i$) is a $t$-element circuit (resp.\ cocircuit) of $M$ containing $X_i$, we denote the unique element in $C_i-X_i$ (resp.\ $C^*_i-X_i$) by $c_i$ (resp.\ $c^*_i$). The proof of Theorem~\ref{main1} is essentially partitioned into a sequence of lemmas.

\begin{lemma}
\label{two}
Let $n\ge 4t-6$. Then
\begin{enumerate}[{\rm (i)}]
\item \blue{there is no $i\in [n]$ such that} $C_i$ and $C_{i+2t-4}$ are $t$-element circuits containing $X_i$ and $X_{i+2t-4}$, respectively, with $C_i\not\subseteq \sigma_{[i,\, i+3t-6]}$ and $C_{i+2t-4}\not\subseteq \sigma_{[i,\, i+3t-6]}$, and

\item \blue{there is no $i\in [n]$ such that} $C^*_i$ and $C^*_{i+2t-4}$ are $t$-element cocircuits containing $X_i$ and $X_{i+2t-4}$, respectively, with $C^*_i\not\subseteq \sigma_{[i,\, i+3t-6]}$ and $C^*_{i+2t-4}\not\subseteq \sigma_{[i,\, i+3t-6]}$.
\end{enumerate}
\end{lemma}

\begin{proof}
We will prove (i). The proof of (ii) is the same except the roles of the circuits and cocircuits are interchanged. Suppose there is some $i\in [n]$ for which~(i) holds. Then there are $t$-element circuits $C_i$ and $C_{i+2t-4}$ containing $X_i$ and $X_{i+2t-4}$, respectively, such that $c_i\not\in \sigma_{[i,\, i+3t-6]}$ and $c_{i+2t-4}\not\in \sigma_{[i,\, i+3t-6]}$. Let $C^*_{i+t-2}$ be a $t$-element cocircuit containing $X_{i+t-2}$. If $c^*_{i+t-2}\in \sigma_{[i,\, i+3t-6]}$, then $C^*_{i+t-2}$ intersects either $C_i$ or $C_{i+2t-4}$ in exactly one element, contradicting orthogonality. So $c^*_{i+t-2}\not\in \sigma_{[i,\, i+3t-6]}$. Therefore, as $C^*_{i+t-2}$ intersects each of the disjoint sets $X_i$ and $X_{i+2t-4}$ in exactly one element, it follows by orthogonality that
$$c_i=c^*_{i+t-2}=c_{i+2t-4}.$$
Now, as $n\ge 4t-6$, there exists an element $j\in [n]-[i, i+3t-6]$ such that $c_i\in X_j$ and $X_j\cap \sigma_{[i,\, i+3t-6]}=\emptyset$. By orthogonality again, this implies that any cocircuit $C^*_j$ containing $X_j$ has the property that $|C^*_j\cap X_i|\neq \emptyset$ and $|C^*_j\cap X_{i+2t-4}|\neq \emptyset$. But this is not possible as $X_i$ and $X_{i+2t-4}$ are disjoint. This contradiction completes the proof of the lemma.
\end{proof}

%
%

The next lemma is the base case for the inductive proof of Lemma~\ref{tight}.

\begin{lemma}
\label{bound}
Let $n\ge 4t-6$. For all $i\in [n]$, if $C_i$ is a $t$-element circuit containing $X_i$ and $C^*_i$ is a $t$-element cocircuit containing $X_i$, then
$$C_i, C^*_i\subseteq \sigma_{[i-(2t-4),\, i+3t-6]}.$$
\end{lemma}

\begin{proof}
Let $i\in [n]$, and suppose that $C_i$ is a $t$-element circuit containing $X_i$. If $C_i\subseteq \sigma_{[i,\, i+3t-6]}$, then $C_i\subseteq \sigma_{[i-(2t-4),\, i+3t-6]}$. Therefore assume that $C_i\not\subseteq \sigma_{[i,\, i+3t-6]}$. By Lemma~\ref{two}, this implies that if $C_{i+2t-4}$ is a $t$-element circuit containing $X_{i+2t-4}$, then $C_{i+2t-4}\subseteq \sigma_{[i,\, i+3t-6]}$, in which case, $c_{i+2t-4}\in \sigma_{[i,\ i+2t-\blue{5}]}$. By Lemma~\ref{two} again, if $C_{i+4t-8}$ is a $t$-element circuit containing $X_{i+4t-8}$, then, as $c_{i+2t-4}\in \sigma_{[i,\, i+2t-\blue{5}]}$, we have $C_{i+4t-8}\subseteq \sigma_{[i+2t-4,\, i+5t-10]}$, in which case, $c_{i+4t-8}\in \sigma_{[i+2t-4,\, i+\blue{4}t-9]}$. By next considering a $t$-element circuit $C_{i+6t-12}$ containing $X_{i+6t-12}$, applying Lemma~\ref{two}, and continuing in this way, we deduce that, for all positive integers $j$, if $C_{i+j(2t-4)}$ is a $t$-element circuit containing $X_{i+j(2t-4)}$, then
$$c_{i+j(2t-4)}\in \sigma_{[i+(j-1)(2t-4),\, i+j(2t-4)-1]}.$$
In particular, choosing $j=n$, we have $i\equiv i+n(2t-4) \bmod{n}$, and so
$$c_i\in \sigma_{[i-(2t-4),\, i-1]}.$$
Hence $C_i\subseteq \sigma_{[i-(2t-4),\, i+t-2]}$, that is,
$$C_i\subseteq \sigma_{[i-(2t-4),\, i+3t-6]}.$$
The proof for when $C^*_i$ is a $t$-element cocircuit containing $X_i$ is the same but with the roles of the circuits and cocircuits interchanged.
\end{proof}

\begin{lemma}
\label{tight}
Let $n\ge 6t-10$. For all $i\in [n]$, if $C_i$ is a $t$-element circuit containing $X_i$ and $C^*_i$ is a $t$-element cocircuit containing $X_i$, then
$$C_i, C^*_i\subseteq \sigma_{[i-1,\, i+t-1]}.$$
\end{lemma}

\begin{proof}
We establish the lemma using induction by showing that, for all $1\le j\le 2t-4$,
$$C_i, C^*_i\subseteq \sigma_{[i-j,\, i+(t-2)+j]}.$$
If $j=2t-4$, then, by Lemma~\ref{bound},
$$C_i, C^*_i\subseteq \sigma_{[i-(2t-4),\, i+(t-2)+(2t-4)]}$$
for all $i\in [n]$. Now suppose that, for all $i\in [n]$,
$$C_i, C^*_i\subseteq \sigma_{[i-(j+1),\, i+(t-2)+(j+1)]},$$
where $1\le j\le 2t-5$. We next show that $C_i\subseteq \sigma_{[i-j,\, i+(t-2)+j]}$. The proof that $C^*_i\subseteq \sigma_{[i-j,\, i+(t-2)+j]}$ is the same except the roles of the circuits and cocircuits are interchanged.

If, for some $i\in [n]$, we have
$$C_i\not\subseteq \sigma_{[i-j,\, i+(t-2)+j]},$$
then, up to reversing the cyclic ordering, we may assume by the induction assumption that $C_i=X_i\cup e_{i+(t-2)+(j+1)}$. Since $t\ge 3$, each of $X_{i+(t-2)+j}$ and $X_{i+(t-2)+(j+1)}$ contains $e_{i+(t-2)+(j+1)}$. But, as $n\ge 6t-10$, each of $X_{i+(t-2)+j}$ and $X_{i+(t-2)+(j+1)}$ has an empty intersection with $X_i$, it follows by orthogonality that if $C^*_{i+(t-2)+j}$ and $C^*_{i+(t-2)+(j+1)}$ are $t$ element cocircuits containing $X_{i+(t-2)+j}$ and $X_{i+(t-2)+(j+1)}$, respectively, then
$$\{c^*_{i+(t-2)+j}, c^*_{i+(t-2)+(j+1)}\}\subseteq X_i.$$
By the induction assumption, $c^*_{i+(t-2)+(j+1)}=e_{i+(t-2)}$ and $c^*_{i+(t-2)+j}\in \{e_{i+(t-3)}, e_{i+(t-2)}\}$. The first of these outcomes implies that $C^*_{i+(t-2)+(j+1)}=X_{i+(t-2)+(j+1)}\cup e_{i+(t-2)}$. Thus if $C_{i+(2t-4)+(j+1)}$ is a $t$-element circuit containing $X_{i+(2t-4)+(j+1)}$, then, as $e_{i+(2t-4)+(j+1)}\in C^*_{i+(t-2)+(j+1)}$ and $n\ge 6t-10$, it follows by orthogonality and the induction assumption,
$$c_{i+(2t-4)+(j+1)}\in X_{i+(t-2)+(j+1)}-e_{i+(2t-4)+(j+1)}.$$
Now $X_{i+(t-2)+(j+1)}-e_{i+(2t-4)+(j+1)}\subseteq C^*_{i+(t-2)+j}$, and so $c_{i+(2t-4)+(j+1)}\in C^*_{i+(t-2)+j}$. But then, as $c^*_{i+(t-2)+j}\in \{e_{i+(t-3)}, e_{i+(t-2)}\}$ and $n\ge 6t-10$, we have
$$|C^*_{i+(t-2)+j}\cap C_{i+(2t-4)+(j+1)}|=1,$$
contradicting orthogonality. Hence, for all $i\in [n]$,
$$C_i\subseteq \sigma_{[i-j,\, i+(t-2)+j]},$$
thereby completing the proof of the lemma.
\end{proof}

The next lemma shows that, for all $i\in [n]$, there is a unique $t$-element circuit containing $X_i$ and a unique $t$-element cocircuit containing $X_i$.

\begin{lemma}
\label{unique}
Let $n\ge 6t-10$ and let $i\in [n]$.
\begin{enumerate}[{\rm (i)}]
\item If $C_i$ and $D_i$ are two $t$-element circuits containing $X_i$, then $C_i=D_i$.

\item If $C^*_i$ and $D^*_i$ are two $t$-element cocircuits containing $X_i$, then $C^*_i=D^*_i$.
\end{enumerate}
\end{lemma}

\begin{proof}
To prove (i), suppose that $C_i\neq D_i$. Then, by Lemma~\ref{tight}, we may assume without loss of generality that $C_i=X_i\cup e_{i+t-1}$ and $D_i=X_i\cup e_{i-1}$. Let $C^*_{i-(t-1)}$ be a $t$-element cocircuit containing $X_{i-(t-1)}$. Since $|X_{i-(t-1)}\cap D_i|=1$, it follows by orthogonality that $c^*_{i-(t-1)}\in D_i-e_{i-1}$. But then $|C^*_{i-(t-1)}\cap C_i|=1$, contradicting orthogonality. Hence $C_i=D_i$. The proof of (ii) is the same but with the roles of the circuits and cocircuits interchanged.
\end{proof}

\begin{lemma}
\label{same}
Let $n\ge 6t-10$.
\begin{enumerate}[{\rm (i)}]
\item For some $i\in [n]$, let $C_i$ be a $t$-element circuit containing $X_i$ and suppose that $C_i=\sigma_{[i,\, i+t-1]}$.
\blue{If} $j\equiv i \bmod{2}$, then $C_j=\sigma_{[j,\, j+t-1]}$ and $C_{j+1}=\sigma_{[j,\, j+t-1]}$, where $C_j$ and $C_{j+1}$ are any $t$-element circuits containing $X_j$ and $X_{j+1}$, respectively.

\item For some $i\in [n]$, let $C^*_i$ be a $t$-element cocircuit containing $X_i$ and suppose that $C^*_i=\sigma_{[i,\, i+t-1]}$.
\blue{If} $j\equiv i\bmod{2}$, then $C^*_j=\sigma_{[j,\, j+t-1]}$ and $C^*_{j+1}=\sigma_{[j,\, j+t-1]}$, where $C^*_j$ and $C^*_{j+1}$ are any $t$-element cocircuits containing $X_j$ and $X_{j+1}$, respectively.
\end{enumerate}
\end{lemma}

\begin{proof}
We will prove (i), as the proof of (ii) is the same except the roles of the circuits and cocircuits are interchanged.
\blue{Since $X_{i+1}\subseteq C_i$, it follows by Lemma~\ref{unique} that} $C_{i+1}=\sigma_{[i,\, i+t-1]}$. \blue{Thus,} by Lemma~\ref{tight}, if $C_{i+2}$ is a $t$-element circuit containing $X_{i+2}$, then $C_{i+2}=\sigma_{[i+2,\, i+t+1]}$, and so $C_{i+3}=\blue{\sigma}_{[i+2,\, i+t+1]}$, where $C_{i+3}$ is any $t$-element circuit containing $X_{i+3}$. Continuing in this way establishes (i).
\end{proof}

\begin{proof}[Proof of Theorem~\ref{main1}]
It immediately follows from Lemmas~\ref{unique} and~\ref{same} that $n$ is even and, for all $i\in [n]$, there is a unique $t$-element circuit and a unique $t$-element cocircuit containing $X_i$. We next establish (I) and (II).

Up to reversing the ordering of $\sigma$, we may assume, by Lemmas~\ref{tight} and~\ref{unique}, that the unique $t$-element circuit containing $X_i$ is $C_i=X_i\cup e_{i+t-1}$. Let $C^*_i$ denote the unique $t$-element cocircuit containing $X_i$. First suppose $t$ is odd. Say
$$C_i=\sigma_{[i,\, i+t-1]}=C^*_i.$$
By Lemma~\ref{same} the unique $t$-element circuit containing $X_{i+t-1}$ is $C_{i+t-1}=\sigma_{[i+t-1,\, i+2t-2]}$. But then $|C^*_i\cap C_{i+(t-1)}|=1$; a contradiction. Thus, by Lemma~\ref{tight}, $C^*_i=\sigma_{[i-1,\, i+t-1]}$. Part~(I) now follows from Lemma~\ref{same}.

Now suppose $t$ is even, and say $C^*_i=\sigma_{[i-1,\, i+t-2]}$. By Lemma~\ref{same}, the unique $t$-element circuit containing $X_i$ is $C_{i+t-1}=\sigma_{[i+t-1,\, i+2t-2]}$. But then
$$|C^*_i\cap C_{i+t-1}|=1,$$
contradicting orthogonality. Therefore, by Lemma~\ref{tight}, $C^*_i=\sigma_{[i,\, i+t-1]}$. Part~(II) immediately follows from Lemma~\ref{same}.
\end{proof}

\section{$t$-Cyclic Matroids}
\label{tcyclic}

Let $M$ be a matroid with $n=|E(M)|$, and let $t$ be a positive integer such that $n\ge t+1$. Recall that $M$ is {\em $t$-cyclic} if there exists a cyclic ordering $\sigma=(e_1, e_2, \ldots, e_n)$ of $E(M)$ such that, for all odd $i\in [n]$, either
\begin{enumerate}[{\rm (i)}]
\item $\{e_i, e_{i+1}, \ldots, e_{i+t-1}\}$ is a $t$-element circuit and $\{e_{i+1}, e_{i+2}, \ldots, e_{i+t}\}$ is a $t$-element cocircuit, or

\item $\{e_i, e_{i+1}, \ldots, e_{i+t-1}\}$ is both a $t$-element circuit and a $t$-element cocircuit.
\end{enumerate}
If $\sigma$ is such an ordering of $E(M)$, then $\sigma$ is called a {\em $t$-cyclic ordering} of $M$. Moreover, $\sigma$ is {\em odd} if it satisfies (i) and {\em even} if it satisfies (ii).

Wheels and whirls with at least four elements are $3$-cyclic matroids. Furthermore, if $M$ is $t$-cyclic for some integer $t\ge 2$, then $M$ has the cyclic $(t-1, t)$-property. In fact, if $t=2$, or $t\ge 3$ and the ground set of $M$ is sufficiently large, then the converse also holds.

\begin{proposition}
\label{coincide}
Let $M$ be a matroid with $n=|E(M)|$, and suppose that $t=2$, or $t\ge 3$ and $n\ge 6t-10$. Then $M$ is $t$-cyclic if and only if it has the cyclic $(t-1, t)$-property.
\end{proposition}

\begin{proof}
Evidently, if $M$ is $t$-cyclic, then $M$ has the cyclic $(t-1, t)$-property. For the converse, if $t=2$ and $M$ has the cyclic $(1, 2)$-property, then $M$ is the direct sum of copies of $U_{1, 2}$, and so a cyclic ordering of $E(M)$ in which the two elements in each copy of $U_{1, 2}$ are consecutive is a $2$-cyclic ordering of $M$. Furthermore, if $t\ge 3$ and $M$ has the cyclic $(t-1, t)$-property, then, as $n\ge 6t-10$, it follows by Theorem~\ref{main1} that $M$ is $t$-cyclic.
\end{proof}

We next establish several basic properties of $t$-cyclic matroids. First note that, by definition, if $M$ is a $t$-cyclic matroid for some $t\ge 1$, then $M^*$ is also $t$-cyclic.

\begin{lemma}
\label{size}
Let $t\ge 1$ and let $M$ be a $t$-cyclic matroid. Then
\begin{enumerate}[{\rm (i)}]
\item $|E(M)|\ge 2t-2$, and

\item $|E(M)|$ is even.
\end{enumerate}
\end{lemma}

\begin{proof}
Let $n=|E(M)|$. We first establish (i). Since $M$ is $t$-cyclic and $n\ge t+1$, it follows that $M$ has a $t$-element circuit and a $t$-element cocircuit. That is, $M$ has an $(n-t)$-element cohyperplane and an $(n-t)$-element hyperplane, and so $r^*(M)-1\le n-t$ and $r(M)-1\le n-t$. Therefore
$$n=r^*(M)+r(M)\le 2n-2t+2.$$
In particular, $n\ge 2t-2$.

To prove (ii), suppose $n$ is odd. Then, regardless of whether $t$ is odd or even, $\{e_1, e_2, \ldots, e_t\}$ is a $t$-element circuit $C$ and $\{e_{n-(t-2)}, e_{n-(t-3)}, \ldots, e_1\}$ is a $t$-element cocircuit $C^*$. By (i), $n\ge 2t-2$ and so, as $n$ is odd, $n\ge 2t-1$. In turn, this implies that $|C\cap C^*|=1$, contradicting orthogonality. It follows that $n$ is even.
\end{proof}

\begin{lemma}
\label{rank}
Let $t\ge 1$, and let $M$ be a $t$-cyclic matroid. Then
$$r(M)=r^*(M)={\textstyle \frac{1}{2}}|E(M)|.$$
\end{lemma}

\begin{proof}
Let $n=|E(M)|$ and let $\sigma=(e_1, e_2, \ldots, e_n)$ be a $t$-cyclic ordering of $M$. By Lemma~\ref{size}, $n\ge 2t-2$ and $n$ is even. First suppose $n=2t-2$. Then, as $\{e_{t-1}, e_t, \ldots, e_{2t-2}\}$ is a cocircuit, it follows that $\{e_1, e_2, \ldots, e_{t-1}\}$ spans $M$ as its complement contains no cocircuit. Similarly, $\{e_2, e_3, \ldots, e_t\}$ cospans $M$. Thus $r(M)\le t-1$ and $r^*(M)\le t-1$, that is $r(M)=r^*(M)=\frac{n}{2}$.

Now suppose that $n\ge 2t$. Since $\{e_{n-(t-1)}, e_{n-(t-2)}, \ldots, e_n\}$ is a $t$-element cocircuit, $Y=\{e_1, e_2, \ldots, e_{n-t}\}$ is a hyperplane of $M$. Moreover, as $\{e_i, e_{i+1}, \ldots, e_{i+t-1}\}$ is a $t$-element circuit for all odd $i\in [n]$, it is easily checked that
$$X=\{e_1, e_2, \ldots, e_{t-1}, e_{t+2}, e_{t+4}, \ldots, e_{n-t}\}$$
spans $Y$. Therefore
$$r(M)-1=r(Y)\le |X|=\frac{n}{2}-1,$$
and so $r(M)\le \frac{n}{2}$. If $t$ is even, then $\{e_{n-(t-1)}, e_{n-(t-2)}, \ldots, e_n\}$ is also a $t$-element circuit, and so an analogous argument shows that $Y$ is a cohyperplane and $X$ cospans $Y$. Thus $r^*(M)\le \frac{n}{2}$, and so $r(M)=r^*(M)=\frac{n}{2}$.

Now assume $t$ is odd. Then $\{e_{n-t}, e_{n-(t-1)}, \ldots, e_{n-1}\}$ is a circuit of $M$, and so
$$Y'=\{e_n, e_1, e_2, \ldots, e_{n-(t+1)}\}$$
is a cohyperplane of $M$. Now let
$$X'=\{e_n, e_1, e_2, \ldots, e_{t-2}, e_{t+1}, e_{t+3}, \ldots, e_{n-(t+1)}\}.$$
Since $\{e_{i+1}, e_{i+2}, \ldots, e_{i+t}\}$ is a $t$-element cocircuit for all odd $i\in [n]$, it is easily checked that $X'$ cospans $Y'$. Thus
$$r^*(M)-1=r^*(Y')\le |X'|=\frac{n}{2}-1,$$
and therefore $r^*(M)\le \frac{n}{2}$. Hence, if $t$ is odd, then $r(M)=r^*(M)=\frac{n}{2}$. This completes the proof of the lemma.
\end{proof}

\section{Flowers}
\label{flowers}

In this section we establish Theorems~\ref{oddflower} and~\ref{evenflower}. We would have liked to prove these theorems simultaneously. However, apart from the first lemma, the cases of when $t$ is odd or even are treated separately to avoid any ambiguity.

Regardless of whether $t$ is even or odd, if $\sigma=(e_1, e_2, \ldots, e_n)$ is a $t$-cyclic ordering of a matroid $M$, then, for all $j\in [n]$, the $(t-1)$-element set $\{e_{j+1}, e_{j+2}, \ldots, e_{j+(t-1)}\}$ is both coindependent and independent. This is because it is properly contained in a $t$-element cocircuit and a $t$-element circuit. The next lemma extends this observation for when $n\ge 2t$.

\begin{lemma}
\label{oddindep}
Let $M$ be a matroid, and let $\sigma=(e_1, e_2, \ldots, e_n)$ be a $t$-cyclic ordering of $E(M)$ for some positive integer $t$, and suppose that $n\ge 2t$.
\begin{enumerate}[{\rm (i)}]
\item If $t$ is odd, then, for all odd $i\in [n]$, we have $\{e_i, e_{i+1}, \ldots, e_{i+t-1}\}$ is coindependent and $\{e_{i+1}, e_{i+2}, \ldots, e_{i+t}\}$ is independent.

\item If $t$ is even, then, for all odd $i\in [n]$, we have $\{e_{i+1}, e_{i+2}, \ldots, e_{i+t}\}$ is both independent and coindependent.
\end{enumerate}
\end{lemma}

\begin{proof}
To prove (i), suppose $t$ is odd and first assume, for some odd $i\in [n]$, that $X=\{e_i, e_{i+1}, \ldots, e_{i+t-1}\}$ is codependent. Then $X$ contains a cocircuit $C$. Now $\{e_{i-(t-1)}, e_{i-(t-2)}, \ldots, e_i\}$ and $\{e_{i+t-1}, e_{i+t}, \ldots, e_{i+2t-2}\}$ are $t$-element circuits and so, as $n\ge 2t$, it follows by orthogonality that $e_i\not\in C$ and $e_{i+t-1}\not\in C$. That is,
$$C\subseteq \{e_{i+1}, e_{i+2}, \ldots, e_{i+t-2}\}.$$
But then $C$ is properly contained in the $t$-element cocircuit $\{e_{i+1}, e_{i+2}, \ldots, e_{i+t-1}\}$; a contradiction. Thus, for all odd $i\in [n]$, the set $\{e_i, e_{i+1}, \ldots, e_{i+t-1}\}$ is coindependent. The proof, for all odd $i\in [n]$, that $\{e_{i+1}, e_{i+2}, \ldots, e_{i+t}\}$ is independent as well as the proof of (ii) is similar and omitted.
\end{proof}


We now work towards proving Theorem~\ref{oddflower}, which applies when $t$ is an odd integer exceeding one. As remarked in the introduction, if $M$ is a $1$-cyclic matroid and $\sigma$ is a $1$-cyclic ordering of $M$, then any concatenation $\Phi$ of $\sigma$ into non-empty sets is a $1$-anemone.

\begin{lemma}
\label{petals1}
Let $M$ be a matroid and let $\sigma=(e_1, e_2, \ldots, e_n)$ be an odd $t$-cyclic ordering of $E(M)$ for some odd integer $t$. Then, for all $i\in [n]$ and $1\le j\le \frac{n}{2}$,
$$\lambda(\{e_{i+1}, e_{i+2}, \ldots, e_{i+j}\})=
\begin{cases}
j, & \mbox{if $j < t-1$; and} \\
t-1, & \mbox{if $j\ge t-1$.}
\end{cases}$$
\end{lemma}

\begin{proof}
Fixing $i\in [n]$, let $X=\{e_{i+1}, e_{i+2}, \ldots, e_{i+j}\}$, where $1\le j\le \frac{n}{2}$. Note that $|X|=j$. We argue by induction that, for all $j$, we have $\lambda(X)=j$ if $j < t-1$ and $\lambda(X)=t-1$ otherwise.

If $1\le j\le t-1$, then, $X$ is both independent and coindependent, so
$$\lambda(X)=r(X)+r^*(X)-|X|=j+j-j=j.$$
Thus we may now assume that $n\ge 2t$. If $j=t$, then, by Lemma~\ref{oddindep}, $X$ is either a coindependent circuit or an independent cocircuit. In both instances,
$$\lambda(X)=|X|-1=t-1.$$
Thus the lemma holds if $1\le j\le t$.
 
Now suppose $t+1\le j\le \frac{n}{2}$ and $\lambda(X-e_{i+j})=t-1$. If $i+j$ is odd, then
$$\{e_{i+j-(t-1)}, e_{i+j-(t-2)}, \ldots, e_{i+j}\}$$
is a circuit and so $e_{i+j}\in \cl(X-e_{i+j})$. Now
$$Y=\{e_{i+j}, e_{i+j+1}, \ldots, e_{i+j+(t-1)}\}$$
is also a circuit, so $e_{i+j}\in \cl(Y-e_{i+j})$. Since $j\le \frac{n}{2}$ and $n\ge 2t$, we also have $|Y\cap (X-e_{i+j})|=0$, and so, by Lemma~\ref{handy}, $e_{i+j}\not\in \cl^*(X-e_{i+j})$. Therefore, by the induction assumption,
\begin{align*}
\lambda(X) & = r(X)+r^*(X)-|X| \\
& = r(X-e_{i+j})+r^*(X-e_{i+j})+1-(|X-e_{i+j}|+1) \\
& = \lambda(X-e_{i+j}) = t-1
\end{align*}
The argument for when $i+j$ is even is the same but with the roles of the circuits and cocircuits interchanged. The lemma now follows by induction.
\end{proof}

\begin{lemma}
\label{petalrank}
Let $t$ be an odd integer exceeding one and let $M$ be a matroid. Suppose that $\sigma=(e_1, e_2, \ldots, e_n)$ is an odd $t$-cyclic ordering of $M$. If $P=\{e_{i+1}, e_{i+2}, \ldots, e_{i+k}\}$, where $|P|\ge t-1$ and $|E(M)-P|\ge t-1$, then
$$r(P)=
\begin{cases}
\frac{1}{2}\left(|P|+t-1\right), & \mbox{if $|P|$ is even;} \\
\frac{1}{2}\left(|P|+t-2\right), & \mbox{if $i+1$ is odd and $|P|$ is odd; and} \\
\frac{1}{2}\left(|P|+t\right), & \mbox{if $i+1$ is even and $|P|$ is odd.}
\end{cases}$$
\end{lemma}

\begin{proof}
We prove the lemma for when $i+1$ is even and $|P|$ is even. The proof for when $i+1$ is odd and $|P|$ is even as well as the other two instances is similar and omitted.

Suppose $i+1$ and $|P|$ are both even. If $|P|=t-1$, then $r(P)=t-1$, so we may assume $|P|\ge t+1$ and, therefore, $n\ge 2t$. Then, by Lemma~\ref{oddindep}, $\{e_{i+1}, e_{i+2}, \ldots e_{i+t}\}$ is independent. Now let $j\in \{t+1, t+2, \ldots, k\}$. If $j$ is even, then $e_{i+j}\in \cl(\{e_{i+j-(t-1)}, e_{i+j-(t-2)}, \ldots, e_{i+j-1}\})$ as $\{e_{i+j-(t-1)}, e_{i+j-(t-2)}, \ldots, e_{i+j}\}$ is a $t$-element circuit. \blue{On the other hand,} if $j$ is odd, then $e_{i+j}\in \cl^*(\{e_{i+j+1}, e_{i+j+2}, \ldots, e_{i+j+t-1}\})$ as $\{e_{i+j}, e_{i+j+1}, \ldots, e_{i+j+t-1}\}$ is a $t$-element cocircuit, and so, as $|E(M)-P|\ge t-1$, it follows by Lemma~\ref{handy} that $e_{i+j}\not\in \cl(\{e_{i+1}, e_{i+2}, \ldots, e_{i+j-1}\})$. By considering each of the elements $e_{i+t+1}, e_{i+t+2}, \ldots, e_{i+k}$ in turn, we deduce that
$$X=\{e_{i+1}, e_{i+2}, \ldots, e_{i+t}, e_{i+t+2}, e_{i+t+4}, \ldots, e_{i+k-1}\}$$
is a basis of $M|P$. As
$$|X|=t-1+{\textstyle \frac{1}{2}}(|P|-(t-1))={\textstyle \frac{1}{2}}(|P|+t-1),$$
the lemma holds when $i+1$ is even and $|P|$ is even.
\end{proof}

We now prove Theorem~\ref{oddflower} which, for convenience, we restate.

\noindent {\bf Theorem~\ref{oddflower}.} Let $t$ be a positive odd integer exceeding one, and let $M$ be a matroid. Suppose that $\sigma$ is an odd t-cyclic ordering of $M$. If $\Phi=(P_1, P_2, \ldots, P_m)$ is a concatenation of $\sigma$ with $|P_i|\ge t-1$ for all $i\in [m]$, then $\Phi$ is a $t$-daisy. Moreover, for all $i\in [m]$, we have $\sqcap(P_i, P_{i+1})=\frac{1}{2}(t-1)$ and, for all non-consecutive petals $P_i$ and $P_j$, we have $\sqcap(P_i, P_j)\le \frac{1}{2}(t-3)$.

\begin{proof}
Let $\sigma=(e_1, e_2, \ldots, e_n)$, and let $\Phi=(P_1, P_2, \ldots, P_m)$ be a concatenation of $\sigma$ with $|P_i|\ge t-1$ for all $i\in [m]$. Since $\lambda$ is symmetric, it follows by Lemma~\ref{petals1} that $\Phi$ is a $t$-flower. To establish that $\Phi$ is a $t$-daisy with the desired local connectivities, it suffices, by \cite[Lemma~4.3]{Aikin2008}, to show that $\sqcap(P_1, P_2)=\frac{1}{2}(t-1)$ if $m\ge 3$ and $\sqcap(P_1, P_3)\le \frac{1}{2}(t-3)$ if $m\ge 4$.

Let $P_1=\{e_{i+1}, e_{i+2}, \ldots, e_{i+k}\}$, and suppose $m\ge 3$. We begin by showing that $\sqcap(P_1, P_2)=\frac{1}{2}(t-1)$. First assume that $i+1$ is odd, \blue{$|P_1|$} is odd, and $|P_2|$ is odd. Then, by Lemma~\ref{petalrank},
\begin{align*}
\sqcap(P_1, P_2) & = r(P_1)+r(P_2)-r(P_1\cup P_2) \\
& = {\textstyle \frac{1}{2}}(|P_1|+t-2)+{\textstyle \frac{1}{2}}(|P_2|+t)-{\textstyle \frac{1}{2}}(|P_1\cup P_2|+t-1) \\
& = {\textstyle \frac{1}{2}}(t-1).
\end{align*}
The remaining cases, which depend on whether $i+1$ is odd or even, $|P_1|$ is odd or even, and $|P_2|$ is odd or even, are also routine and omitted. Hence $\sqcap(P_1, P_2)=\frac{1}{2}(t-1)$.

Now let $P_3=\{e_{j+1}, e_{j+2}, \ldots, e_{j+\ell}\}$, and suppose $m\ge 4$. To show that $\sqcap(P_1, P_3)\le \frac{1}{2}(t-3)$, we first establish that
\begin{align}\label{lower}
r(P_1\cup P_3)\ge
\begin{cases}
r(P_1)+\frac{1}{2}(|P_3|+1), & \mbox{if $j+1$ odd;} \\
r(P_1)+\frac{1}{2}(|P_3|+2), & \mbox{if $j+1$ even and $|P_3|$ is even;} \\
r(P_1)+\frac{1}{2}(|P_3|+3), & \mbox{if $j+1$ is even and $|P_3|$ is odd.}
\end{cases}
\end{align}
We prove the inequality for when $j+1$ is even. The result for when $j+1$ is odd is similar, but slightly more straightforward, and is omitted. If $j+1$ is even, then $e_{j+2}\in \cl^*(\{e_{j+1}, e_{j+3}, \ldots, e_{j+t}\}$ and so, as $|P_4|\ge t-1$, by Lemma~\ref{handy}, $e_{j+2}\not\in \cl(P_1)$. Thus $P_1\cup e_{j+2}$ is independent. Furthermore, since $e_{j+1}\in \cl^*(\{e_{j-(t-2)}, e_{j-(t-3)}, \ldots, e_j\})$ and $|P_2|\ge t-1$, it follows by Lemma ~\ref{handy} that $e_{j+1}\not\in \cl(P_1\cup e_{j+2})$. Therefore $P_1\cup \{e_{j+1}, e_{j+2}\}$ is independent. Repeatedly using Lemma~\ref{handy} and the fact that $P_1$ and $P_3$ are non-consecutive and $|P_4|\ge t-1$, it is easily seen that
$$P_1\cup \{e_{j+1}, e_{j+2}, e_{j+3}, e_{j+5}, \ldots, e_{j+\ell-2}, e_{j+\ell}\}$$
is independent if $|P_3|$ is odd and
$$P_1\cup \{e_{j+1}, e_{j+2}, e_{j+3}, e_{j+5}, \ldots, e_{j+\ell-3}, e_{j+\ell-1}\}$$
is independent if $|P_3|$ is even.
Since
$$|\{e_{j+1}, e_{j+2}, e_{j+3}, e_{j+5}, \ldots, e_{j+\ell-2}, e_{j+\ell}\}|={\textstyle \frac{1}{2}}(|P_3|+3)$$
and
$$|\{e_{j+1}, e_{j+2}, e_{j+3}, e_{j+5}, \ldots, e_{j+\ell-3}, e_{j+\ell-1}\}|={\textstyle \frac{1}{2}}(|P_3|+2),$$
we have $r(P_1\cup P_3)\ge r(P_1)+\frac{1}{2}(|P_3|+3)$ if $|P_3|$ is odd and $r(P_1\cup P_3)\ge r(P_1)+\frac{1}{2}(|P_3|+2)$ if $|P_3|$ is even. It follows that~(\ref{lower}) holds.

Next consider $\sqcap(P_1, P_3)$. If $j+1$ is odd and $|P_3|$ is odd, then, by Lemma~\ref{petalrank} and~(\ref{lower}),
\begin{align*}
\sqcap(P_1, P_3) & = r(P_1)+r(P_3)-r(P_1\cup P_3) \\
& \le r(P_1)+{\textstyle \frac{1}{2}}(|P_3|+t-2)-\left(r(P_1)+{\textstyle \frac{1}{2}}(|P_3|+1)\right) \\
& = {\textstyle \frac{1}{2}}(t-3).
\end{align*}
The remaining three cases are similarly checked. This completes the proof of the theorem.
\end{proof}

The proof of Theorem~\ref{evenflower} takes the same approach as the proof of Theorem~\ref{oddflower}.

\begin{lemma}
\label{petals2}
Let $M$ be a matroid and let $\sigma=(e_1, e_2, \ldots, e_n)$ be an even $t$-cyclic ordering of $E(M)$ for some even integer $t$. Then, for all $i\in [n]$ and $1\le j\le \frac{n}{2}$,
$$\lambda(\{e_{i+1}, e_{i+2}, \ldots, e_{i+j}\})=
\begin{cases}
j, & \mbox{if $j\le t-1$;} \\
t-2, & \mbox{if $j > t-1$, $j$ is even, $i$ is even;} \\
t-1, & \mbox{if $j > t-1$, $j$ is odd;} \\
t, & \mbox{if $j > t-1$, $j$ is even, $i$ is odd.}
\end{cases}$$
\end{lemma}

\begin{proof}
Fixing $i\in [n]$, let $X=\{e_{i+1}, e_{i+2}, \ldots, e_{i+j}\}$, where $1\le j\le \frac{n}{2}$. Note that $|X|=j$. We establish the proof by showing that $\lambda(X)$ has the desired value for all $1\le j\le \frac{n}{2}$ using induction on $j$. If $1\le j\le t-1$, then $X$ is both independent and coindependent, so
$$\lambda(X)=r(X)+r^*(X)-|X|=j+j-j=j.$$
Hence the lemma holds if $1\le j\le t-1$.

Now suppose that $t\le j\le \frac{n}{2}$, in which case $n\ge 2t$, and $\lambda(\{e_{i+1}, e_{i+2}, \ldots, e_{i+j-1}\})$ has the desired value. First, assume both $i$ and $j$ are even. Then
$$\{e_{i+j-(t-1)}, e_{i+j-(t-2)}, \ldots, e_{i+j}\}$$
is both a circuit and a cocircuit, and so $e_{i+j}\in \cl(X-e_{i+j})$ and $e_{i+j}\in \cl^*(X-e_{i+j})$. By the induction assumption, $\lambda(X-e_{i+j})=t-1$ as $j-1$ is odd. Note that $\lambda(X-e_{i+j})=t-1$ if $j=t$. So
\begin{align*}
\lambda(X) & = r(X)+r^*(X)-|X| \\
& = r(X-e_{i+j})+r^*(X-e_{i+j})-(|X|-1)-1 \\
& = \lambda(X-e_{i+j})-1 = t-2.
\end{align*}

Second, assume $j$ is odd. If $i$ is odd, then $\{e_{i+j-(t-1)}, e_{i+j-(t-2)}, \ldots, e_{i+j}\}$ is both a circuit and a cocircuit. Therefore, $e_{i+j}\in \cl(X-e_{i+j})$ and $e_{i+j}\in \cl^*(X-e_{i+j})$. By the induction assumption, $\lambda(X-e_{i+j})=t$ as $j-1$ is even and $i$ is odd, and $j\neq t$. So
\begin{align*}
\lambda(X) & = r(X)+r^*(X)-|X| \\
& = r(X-e_{i+j})+r^*(X-e_{i+j})-(|X|-1)-1 \\
& = \lambda(X-e_{i+j})-1 = t-1.
\end{align*}
If $i$ is even, then $\{e_{i+j}, e_{i+j+1}, \ldots, e_{i+j+t-1}\}$ is both a circuit and a cocircuit. Therefore $e_{i+j}\in \cl(\{e_{i+j+1}, e_{i+j+2}, \ldots, e_{i+j+t-1}\})$ and $e_{i+j}\in \cl^*(\{e_{i+j+1}, e_{i+j+2}, \ldots, e_{i+j+t-1}\})$. Since $j\le \frac{n}{2}$ and $n\ge 2t$, the set $\{e_{i+j+1}, e_{i+j+2}, \ldots, e_{i+j+t-1}\}$ has an empty intersection with $X-e_{i+j}$, and so, by Lemma~\ref{handy}, $e_{i+j}\not\in \cl^*(X-e_{i+j})$ and $e_{i+j}\not\in \cl(X-e_{i+j})$. By the induction assumption, $\lambda(X-e_{i+j})=t-2$, as $i$ is even, $j-1$ is even, and $j\neq t$. Therefore
\begin{align*}
\lambda(X) & = r(X)+r^*(X)-|X| \\
& = r(X-e_{i+j})+1+r^*(X-e_{i+j})+1-(|X|-1)-1 \\
& = \lambda(X-e_{i+j})+1 = t-1.
\end{align*}

Lastly, assume $j$ is even and $i$ is odd. Then $Y=\{e_{i+j}, e_{i+j+1}, \ldots, e_{i+j+t-1}\}$ is a circuit and a cocircuit, and so $e_{i+j}\in \cl(Y-e_{i+j})$ and $e_{i+j}\in \cl^*(Y-e_{i+j})$. Since $j\le \frac{n}{2}$ and $n\ge 2t$, the set $Y-e_{i+j}$ has an empty intersection with $X-e_{i+j}$. Therefore, by Lemma~\ref{handy}, $e_{i+j}\not\in \cl^*(X-e_{i+j})$ and $e_{i+j}\not\in \cl(X-e_{i+j})$. By the induction assumption, $\lambda(X-e_{i+j})=t-1$ as $j-1$ is odd. Again note that $\lambda(X-e_{i+j})=t-1$ if $j=t$. Thus
\begin{align*}
\lambda(X) & = r(X)+r^*(X)-|X| \\
& = r(X-e_{i+j})+1+r^*(X-e_{i+j})+1-(|X|-1)-1 \\
& = \lambda(X-e_{i+j})+1 = t.
\end{align*}
The lemma now follows.
\end{proof}

\begin{lemma}
\label{petalrank2}
Let $t$ be an even positive integer, let $M$ be a matroid, and suppose that $\sigma=(e_1, e_2, \ldots, e_n)$ is an even $t$-cyclic ordering of $M$. If $P=\{e_{i+1}, e_{i+2}, \ldots, e_{i+k}\}$, where $i+1$ is odd, $|P|$ is even, $|P|\ge t-2$, and $|E(M)-P|\ge t-2$, then
$$r(P)={\textstyle \frac{1}{2}}(|P|+t-2).$$
\end{lemma}

\begin{proof}
If $|P|=t-2$ or $|P|=t$, then $r(P)=t-2$ or $r(P)=t-1$, respectively. Thus we may assume that $|P|\ge t+2$, and so $n\ge 2t$. Then, by Lemma~\ref{oddindep}, \blue{$\{e_{i+2}, e_{i+3}, \ldots, e_{i+t+1}\}$} is independent. Observe that $e_{i+1}\in \cl(\{e_{i+2}, e_{i+3}, \ldots, e_{i+t}\})$. Now let $j\in \{t+2, t+3, \ldots, k\}$. If $j$ is even, then $e_{i+j}\in \cl(\{e_{i+j-(t-1)}, e_{i+j-(t-2)}, \ldots, e_{i+j-1}\})$. On the other hand, if $j$ is odd, then
$$e_{i+j}\in \cl^*(\{e_{i+j+1}, e_{i+j+2}, \ldots, e_{i+j+t-1}\})$$
as $\{e_{i+j}, e_{i+j+1}, \ldots, e_{i+j+t-1}\}$ is a $t$-element cocircuit. Since $j$ is odd and $|P|$ is even, $e_{i+j+1}\in P$ and so, as $|E(M)-P|\ge t-2$, it follows by Lemma~\ref{handy} that $e_{i+j}\not\in\cl(\{e_{i+1}, e_{i+2}, \ldots, e_{i+j-1}\})$. Considering each of the elements $e_{i+t+2}, e_{i+t+3}, \ldots, e_{i+k}$ in turn, we deduce that
$$X=\{e_{i+2}, e_{i+3}, \ldots, e_{i+t+1}, e_{i+t+3}, e_{i+t+5}, \ldots, e_{i+k-1}\}$$
is a basis of $M|P$. Since $|X|=\frac{1}{2}(|P|+t-2)$, the lemma holds.
\end{proof}

For convenience, we restate Theorem~\ref{evenflower}.

\noindent {\bf Theorem~\ref{evenflower}.} Let $t$ be a positive even integer and let $M$ be a matroid. Let $\sigma=(e_1, e_2, \ldots, e_n)$ be an even $t$-cyclic ordering of $E(M)$, and suppose that $\Phi=(P_1, P_2, \ldots, P_m)$ is a concatenation of $\sigma$ such that, for all $i\in [m]$, if
$$P_i=\{e_{j+1}, e_{j+2}, \ldots, e_{j+k}\},$$
then $|P_i|\ge t-2$, $|P_i|$ is even, and $j+1$ is odd. Then $\Phi$ is a $(t-1)$-flower. Moreover, for all $i\in [n]$, we have $\sqcap(P_i, P_{i+1})=\frac{1}{2}(t-2)$.

\begin{proof}
Suppose that $\Phi=(P_1, P_2, \ldots, P_m)$ is a concatenation of $\sigma$ as described in the statement of the theorem. Since $\lambda$ is symmetric, it follows by Lemma~\ref{petals2} that $\Phi$ is a $(t-1)$-flower. To see $\sqcap(P_1, P_2)=\frac{1}{2}(t-2)$ if $m\ge 3$, observe that, by Lemma~\ref{petalrank2},
\begin{align*}
\sqcap(P_1, P_2) & = r(P_1)+r(P_2)-r(P_1\cup P_2) \\
& = {\textstyle \frac{1}{2}}(|P_1|+t-2)+{\textstyle \frac{1}{2}}(|P_2|+t-2)-{\textstyle \frac{1}{2}}(|P_1|+|P_2|+t-2) \\
& = {\textstyle \frac{1}{2}}(t-2).
\end{align*}
This completes the proof of the theorem.
\end{proof}

\section{Construction}
\label{build}

In this section we describe a construction which, for all positive integers $t$ exceeding one, takes a $t$-cyclic matroid and produces a $(t+2)$-cyclic matroid having the same ground set. Let $M$ be a $t$-cyclic matroid with $n=|E(M)|$, where $t\ge 2$ and $n\ge 2(t+2)-2$, and let $\sigma=(e_1, e_2, \ldots, e_n)$ be a $t$-cyclic ordering of $M$. We require that $n\ge 2(t+2)-2$, as a $(t+2)$-cyclic matroid has at least $2(t+2)-2$ elements, by \cref{size}. Let $M'$ be the truncation of $M$. That is, $M'$ is obtained by freely adding an element, $f$ say, to $M$ to get $M_1$ and then contracting $f$ from $M_1$ to get $M'$. For all $j\in [n]$, if $\{e_{j+1}, e_{j+2}, \ldots, e_{j+t}\}$ and $\{e_{j+3}, e_{j+4}, \ldots, e_{j+t+2}\}$ are $t$-element cocircuits of $M$, then $\{e_{j+1}, e_{j+2}, \ldots, e_{j+t+2}\}$ is a $(t+2)$-element cocircuit of $M'$. To see this, it is easily checked that
$$(E(M)-\{e_{j+1}, e_{j+2}, \ldots, e_{j+t+2}\})\cup \{f\}$$
is a hyperplane of $M_1$, so $E(M)-\{e_{j+1}, e_{j+2}, \ldots, e_{j+t+2}\}$ is a hyperplane of $M'$. In other words, $\{e_{j+1}. e_{j+2}, \ldots, e_{j+t+2}\}$ is a cocircuit of $M'$. Next, we let $N$ be the Higgs lift of $M'$. That is, let $M'_1$ be the matroid obtained by freely coextending $M'$ by an element, $g$ say. Observe that $(M'_1)^*$ is the free extension of $(M')^*$. Let $N$ be the matroid obtained from $M'_1$ by deleting $g$. Then, dually, for all $j\in [n]$, if $\{e_{j+1}, e_{j+2}, \ldots, e_{j+t}\}$ and $\{e_{j+3}, e_{j+4}, \ldots, e_{j+t+2}\}$ are $t$-element circuits of $M$, and therefore of $M'$, then $\{e_{j+1}, e_{j+2}, \ldots, e_{j+t+2}\}$ is a $(t+2)$-element circuit of $N$. Hence, $N$ is a $(t+2)$-cyclic matroid. Observe that $\sigma$ is a $(t+2)$-cyclic ordering of $N$.


Let $t$ be an even positive integer exceeding two. We next use this construction to show that, for all $t\ge 4$, there exist $t$-cyclic matroids giving rise to $(t-1)$-anemones and $t$-cyclic matroids giving rise to $(t-1)$-daisies. Let $M$ be a $t$-cyclic matroid, and suppose that $\sigma=(e_1, e_2, \ldots, e_n)$ is an even $t$-cyclic ordering of $E(M)$. We call a concatenation $\Phi=(P_1, P_2, \ldots, P_m)$ of $\sigma$ {\em even} if, for all $i\in [m]$, the set $P_i=\{e_{j+1}, e_{j+2}, \ldots, e_{j+k}\}$ satisfies $|P_i|\ge t-2$, $|P_i|$ is even, and $j+1$ is odd.

Now let $M$ be a rank-$r$ spike, where $r\ge 3$, and let $(L_1, L_2, \ldots, L_r)$ be a partition of the ground set of $M$ into pairs such that, for all distinct $i, j\in \{1, 2, \ldots, r\}$, the union $L_i\cup L_j$ is a $4$-element circuit and $4$-element cocircuit. Then the cyclic ordering $\sigma$ of $E(M)$ in which, for all $i$, the two elements in $L_i$ are consecutive in $\sigma$ is a $4$-cyclic ordering of $M$. Thus $M$ is $4$-cyclic. Furthermore, by Theorem~\ref{evenflower}, any even concatenation $\Phi=(P_1, P_2, \ldots, P_m)$ of $\sigma$ is a $3$-flower and, as $\sqcap(P_1, P_3)=1$, it follows that $\Phi$ is a $3$-anemone.

For a $4$-cyclic matroid giving rise to a $3$-daisy, let $M$ a rank-$r$ swirl, where $r\ge 3$, and let $(L_1, L_2, \ldots, L_r)$ be a partition of the ground set of $M$ into pairs such that $L_i\cup L_{i+1}$ is a $4$-element circuit and a $4$-element cocircuit for all $i$. By choosing $\sigma$ to be a cyclic ordering of $E(M)$ such that $(L_1, L_2, \ldots, L_r)$ is a concatenation of $\sigma$, it follows that $\sigma$ is a $4$-cyclic ordering of $E(M)$, and so $M$ is $4$-cyclic. By Theorem~\ref{evenflower}, any even concatenation $\Phi=(P_1, P_2, \ldots, P_m)$ of such a $\sigma$ is a $3$-flower. To see that $\Phi$ is a $3$-daisy if $m\ge 4$, observe that $\sqcap(P_1, P_3)=0$.

Now suppose that $M$ is a $t$-cyclic matroid with at least $2(t+2)-2$ elements and let $\sigma$ be a $t$-cyclic ordering of $E(M)$. Let $N$ be the matroid obtained from $M$ by the construction detailed above. Then $N$ is a $(t+2)$-cyclic matroid and $\sigma$ is a $(t+2)$-cyclic ordering of $N$. Let $\Phi=(P_1, P_2, \ldots, P_m)$ be an even concatenation of $\sigma$, where $|P_i|\ge t$ for all $i\in [m]$. By Theorem~\ref{evenflower}, $\Phi$ is a $(t-1)$-flower of $M$ and $(t+1)$-flower of $N$. Assume $m\ge 4$, and let $P_i$ and $P_j$ be petals of $\Phi$. Since $|P_i|, |P_j|\ge t$ and so $r_M(P_i\cup P_j)\neq r(M)$, it follows by construction that $r_N(P_i)=r_M(P_i)+1$ and $r_N(P_i\cup P_j)=r_M(P_i\cup P_j)+1$. Hence if $\Phi$ is a $(t-1)$-anemone or a $(t-1)$-daisy of $M$, then $\Phi$ is a $(t+1)$-anemone or a $(t+1)$-daisy of $N$, respectively. The obvious induction gives the desired outcome.

The described construction is a specific example of an operation by which we can obtain a $(t+2)$-cyclic matroid from a $t$-cyclic matroid.
More generally, we can replace the truncation with any elementary quotient such that none of the $t$-element cocircuits corresponding to consecutive elements in the cyclic ordering are preserved; and we can replace the Higgs lift with any elementary lift such that none of the $t$-element circuits corresponding to consecutive elements in the cyclic ordering are preserved. For a $t$-cyclic matroid $M$ with $|E(M)| \ge 2t +2$, we say that $N$ is an \emph{inflation} of $M$ if we can obtain $N$, starting from $M$, by such an elementary quotient, followed by such an elementary lift.  We conjecture the following:

\begin{conjecture}
Let $t$ be an integer exceeding two, and let $M$ be a $t$-cyclic matroid. 
\begin{enumerate}[{\rm(i)}]
  \item If $t$ is even, then $M$ can be obtained from a spike or a swirl by a sequence of inflations.
  \item If $t$ is odd, then $M$ can be obtained from a wheel or whirl by a sequence of inflations.
\end{enumerate}
\end{conjecture}


\section*{Acknowledgements}

\blue{We thank James Oxley for some helpful discussions during the initial stages of this work.}

\bibliographystyle{abbrv}
\bibliography{library}

\end{document}